\documentclass{birkau}
\usepackage{amsmath,amssymb,latexsym,url}
\numberwithin{equation}{section}

\theoremstyle{plain}
\newtheorem{theorem}{Theorem}[section]
\newtheorem{lemma}[theorem]{Lemma}

\newtheorem{corollary}[theorem]{Corollary}

\theoremstyle{definition}

\newtheorem{remark}[theorem]{Remark}

\RequirePackage[mathscr]{euscript}
\usepackage{xcolor}
\usepackage{enumerate}

\usepackage{comment}

\usepackage{tikz}
\usetikzlibrary{arrows,calc,through,intersections,backgrounds,trees,positioning}
\usetikzlibrary{decorations,decorations.pathreplacing,cd,shapes.geometric}

\newcommand{\Con}[1]{\alg{Con}\;\alg #1}        

\newcommand{\iso}{\DOTSB\cong}                  
\newcommand{\comp}{\mathbin\circ}               
\newcommand{\onto}{\twoheadrightarrow}

\DeclareRobustCommand\emslb{\bfseries\slshape}

\DeclareTextFontCommand{\emphslb}{\emslb}

\newcommand{\HH}{\textit{\emphslb{H}}}
\renewcommand{\SS}{\textit{\emphslb{S}}}
\newcommand{\PP}{\textit{\emphslb{P}}}
\newcommand{\VV}{\textit{\emphslb{V}}}

\newcommand{\cov}{\prec}                
\def\join{\DOTSB\vee}                   

\newcommand{\meet}{\DOTSB\wedge}

\newcommand{\la}{\langle}            
\newcommand{\ra}{\rangle}            

\newcommand{\alg}[1]{\mathbf{#1}}
\newcommand{\op}{\operatorname}

\begin{document}
\date{Nov 13, 2023}
\title{Finitely Based Congruence Varieties}
\corrauthor{Ralph Freese}
\address{Department of Mathematics\\
University of Hawaii\\Honolulu, HI 96822\\USA}
\urladdr{http://math.hawaii.edu/~ralph}
\email{ralph@math.hawaii.edu}

\author{Paolo Lipparini}
\address{Department of Mathematics\\University of Tor Vergata\\Viale della Ricerca Scientivica\\I-00133 Rome\\ITALY}
\email{lipparin@axp.mat.uniroma2.it}

\dedicatory{Dedicated to the memory of George F. McNulty}


\begin{abstract}
We show that for a large class of varieties of algebras, the 
equational theory of the congruence lattices of the members 
is not finitely based.
\end{abstract}

\maketitle

\noindent
The Version of Record of this article is published in Algebra Universalis, and is available online 
at https://doi.org/10.1007/s00012-023-00840-6.
Also see the first author's website given at the end of this manuscript.

\section{Introduction}\label{sec:intro}

Let $\mathscr V$ be a variety of algebras and let
\begin{equation}\label{eq:conv}
\alg{Con}(\mathscr V) = \{\alg{Con}(\alg A) : \alg A \in \mathscr V\}.
\end{equation}
The variety of lattices, $\VV\,\alg{Con}(\mathscr V)$,
generated by the congruence lattices of the 
members of $\mathscr V$, is called the \emph{congruence
variety} associated with $\mathscr V$. Congruence varieties
originated with Nation in his 
thesis~\cite{Nation1973}; he  
showed, among other things, that the lattice variety 
generated by
$\mathbf N_5$ (the 5 element nonmodular lattice) is not
a congruence variety; 
see~\cite[Theorem 6.99]{FreeseMcKenzieMcNultyTaylor2022}.

A lattice is \emph{meet semidistributive} if it satisfies the
(universally quantified) implication
\[
x\meet y = x \meet z \,\rightarrow\, x\meet y = x\meet (y \join z)
\]
It is  \emph{join semidistributive} if it satisfies the dual
condition and it is \emph{semidistributive} if it satisfies both.

In \cite{FreeseJonsson1976} Freese and J\'onsson proved that every
modular congruence variety actually satisfies the arguesian identity.
Since the arguesian identity is properly stronger than the modular law (as
witnessed by the lattice of subspaces of any nonarguesian projective plane),
this implies, for example, that the variety of all modular lattices
is not a congruence variety. That the variety of all arguesian lattices
is not a congruence variety was shown in \cite{FreeseHerrmannHuhn1981}.

In \cite[Problem 9.12]{Jonsson1982:b} B.~J\'onsson asked 
if any nontrivial congruence variety could be finitely based
other than the variety of distributive lattices and the variety of all lattices. 
For congruence modular varieties this question was completely answered
by the first author
with the following theorem.

\begin{theorem}[{\cite[Theorem 4]{Freese1994}}]\label{thm:freese94}
 There is no nontrivial 
finitely based modular congruence variety other than the variety 
of distributive lattices. 
\end{theorem}

In this paper we come close to answering J\'onsson's problem 
by extending Theorem~\ref{thm:freese94} with the following theorem:

\begin{theorem}\label{thm:fbcv}
Let $\mathscr V$ be a variety of algebras 
such that 
$\Con(\mathscr V)$ satisfies a nontrivial lattice identity.
Then, if\/ $\Con(\mathscr V)$ has a finite basis for
its equations, $\alg{Con}(\mathscr V)$ is semidistributive.
\end{theorem}

We now outline how we prove Theorem~\ref{thm:fbcv}. 
For each field $\mathbf F$ with at least 3 elements
Haiman~\cite{Haiman1991} has constructed a sequence of modular
lattices $\mathbf H_n(\mathbf F)$, $n\ge 3$. When $n\ge 4$ these lattices are Aguresian
but cannot be represented as lattices of permutable 
equivalence relations. In proving Theorem~\ref{thm:freese94} we showed
that for every $n \ge 3$ and every field $\mathbf F$ with $|\mathbf F| > 2$,

\begin{enumerate}[\quad\normalfont (1)]
 \item
 $\mathbf H_n(\mathbf F)$ lies in no modular congruence variety.
 \item
 For any modular, nondistributive congruence variety $\mathscr K$ there is 
 a field $\mathbf F$ such that a nonprincipal ultraproduct 
 of the $\mathbf H_n(\mathbf F)$'s is in $\mathscr K$.
\end{enumerate}
To prove Theorem~\ref{thm:fbcv} we strengthen these statements as follows. 
By a {\it proper}
congruence variety we mean one that is not the variety of
all lattices.
\begin{theorem}\label{thm:main}
Let $\mathbf H_n(\mathbf F)$ be Haiman's lattices for $\mathbf F$, $|\mathbf F| >2$, 
and  $n \ge 3$.
Then 
\begin{enumerate}[\quad \normalfont(1$'$)]
 \item
 $\mathbf H_n(\mathbf F)$ lies in no congruence variety, 
 except the variety of all lattices.
 \item
 For any congruence variety $\mathscr K$ that is not join semidistributive,
 there is 
 a field $\mathbf F$ such that every nonprincipal ultraproduct 
 of the $\mathbf H_n(\mathbf F)$'s is in~$\mathscr K$.
\end{enumerate}
\end{theorem}

That these statements imply Theorem~\ref{thm:fbcv} is standard;
see~\cite[Theorem~8.52]{FreeseMcKenzieMcNultyTaylor2022}.

The ideas proving (2$'$) also yield interesting results about embedding lattices into members of
$\alg {Con}(\mathscr V)$ for $\mathscr V$ a variety with a weak difference term
which is not congruence meet semidistributive. Such varieties admit a large 
class of modular lattices, as is shown in \S\ref{sec:embed}.

\section{Preliminaries}

Before we begin the proof of (1$'$) and (2$'$) we prove some basic facts 
and introduce some notation. First we will (usually) use $\mathscr V$ to denote
a variety of algebras. $\mathscr K$ denotes a variety of lattices, usually
a congruence variety. This convention is essentially the opposite of the one
used in~\cite{Freese1994}, but is more standard now-a-days.
We let $\HH$, $\SS$, $\PP$, and $\VV = \HH\SS\PP$ be the usual class
operators as defined in 
\S 4.10 
of~\cite{McKenzieMcNultyTaylor1987}.
As mentioned in \eqref{eq:conv} above, for $\mathscr V$ a class of algebras, we let 
$\alg {Con}(\mathscr V)$ simply be the class of all congruence lattices of members
of $\mathscr V$ and use the class operators for the congruence variety:
$\VV\,\alg {Con}(\mathscr V) = \HH\SS\PP\,\alg {Con}(\mathscr V)$,
and also for the \textit{congruence prevariety} $\SS\PP\,\alg {Con}(\mathscr V)$.
Note by the next lemma that the congruence variety is
$\HH\SS\,\alg {Con}(\mathscr V)$
and the congruence prevariety is $\SS\,\alg {Con}(\mathscr V)$.

In most of our results we assume that $\mathscr V$ is a variety of
algebras but in almost all cases it is enough to assume it is closed
under $\SS$ and $\PP$.

\begin{lemma}\label{ConV}
For every variety $\mathscr V$ of algebras, the following hold.
\begin{enumerate}[\normalfont (i)] 
\item $\PP\,\Con(\mathscr V) \subseteq \SS\,\Con(\mathscr V)$.
\item $\VV\,\,\Con(\mathscr V) = \HH\,\SS\,\Con(\mathscr V)$.
\item If $\mathscr U$ is the idempotent reduct of $\mathscr V$, then
$\VV\,\,\Con(\mathscr U)= \VV\,\,\Con(\mathscr V) $.
\end{enumerate}
\end{lemma} 

\begin{proof} 
(i) is
the first statement in Proposition 4.5 of
\cite{FreeseMcKenzie1987} (notice that the proof 
does not require congruence modularity).

(ii) is immediate from (i).

(iii) follows from the well known Pixley-Wille algorithm \cite{Pixley1972,Wille1970},
which shows the property  that some variety
of algebras satisfies a given congruence identity is equivalent to
a weak Mal'cev condition witnessed by idempotent
terms: see for example 
\cite[Theorem 6.111]{FreeseMcKenzieMcNultyTaylor2022}.
\end{proof}

\subsection*{The centrality relation and the commutator}
The \emph{centrality relation} on an algebra $\alg A$ is denoted $\op C (\alpha, \beta; \delta)$
for congruences on $\alg A$; 
see \cite[Definition 11.3 and Lemma 11.4]{FreeseMcKenzieMcNultyTaylor2022:a}
for its definition and basic properties. 
We say $\alpha$ \emph{centralizes $\beta$ modulo $\delta$} whenever the relation
$\op C (\alpha, \beta; \delta)$ holds.

The (term condition) \emph{commutator}, 
$[\alpha,\beta]$, is defined as the least $\delta$ such that 
$\op C(\alpha,\beta;\delta)$.
The centrality relation and the commutator have become a standard tool
in universal algebra. $\alg M_3$ denotes the five element modular, nondistributve lattice. 

\begin{lemma}\label{lemma:M3}
 Suppose $\delta \le \alpha, \beta, \gamma \le \mu$ are congruences on an 
 algebra $\alg A$ that form a copy of $\alg M_3$.
 Then 
 $[\alpha,\alpha] \le \delta$,
 $[\beta,\beta] \le \delta$, and 
 $[\gamma,\gamma] \le \delta$.
\end{lemma}

\begin{proof}
 By  \cite[Lemma~11.4(vi)]{FreeseMcKenzieMcNultyTaylor2022:a}), 
 $C(\alpha,\gamma;\alpha\meet\gamma)$ holds. Since $\delta = \alpha\meet\gamma$, 
 we have $C(\alpha,\gamma;\delta)$.
 Similarly $C(\beta,\gamma;\delta)$. By \cite[Lemma~11.4(iv) and (i)]{FreeseMcKenzieMcNultyTaylor2022:a})
 $\op C(\alpha \join\beta, \gamma; \delta) = \op C(\mu , \gamma; \delta)$, and so 
 $\op C(\gamma, \gamma; \delta)$.  Using the definition 
 of the commutator, this yields $[\gamma,\gamma] \le \delta$. The other two 
 inequalities follow from symmetry.
\end{proof}

\subsection*{Weak difference terms}
Let $\alg A$ be an algebra. A term $d(x,y,z)$ in the signature of $\alg A$
is a \emph{weak difference term} for $\alg A$ if, for all $a$, $b\in A$ and all
$\theta\in \Con(\alg A)$ with $\la a, b \ra \in \theta$,
\begin{equation}
a \mathrel{[\theta,\theta]} d(a, b, b) \quad\text{ and }\quad d(a, a, b) \mathrel{[\theta,\theta]} b.
\end{equation}
Most of the properties 
of the modular commutator hold in varieties with a weak difference term. The next lemma is an illustration.
It is special case of more general permutability results given
in \cite{Lipparini1994}; see for example Theorem~3.5(i) of that paper.

\begin{lemma}\label{lem:M3-permute}
 Let $\alg A$ be an algebra with a weak difference term $d$, and let 
 $\alpha$ and $\beta\in \Con(\alg A)$. Then.
 \begin{enumerate}[\quad \normalfont(1)]
 \item if $[\alpha,\alpha]\le \beta$ and $[\beta,\beta]\le \alpha$, then
 $\alpha$ and $\beta$ permute;
 \item if $\alpha$ and $\beta$ are atoms of a sublattice of\/ $\Con(\alg A)$
 isomorphic to $\alg M_3$, 
 then $\alpha$ and $\beta$ permute.
 \end{enumerate}
\end{lemma}

\begin{proof}
 For (1) suppose $a \mathrel\alpha b \mathrel \beta c$. Then
 \[
 a \mathrel{[\alpha,\alpha]} d(a, b, b) \mathrel\beta d(a, b, c)
 \]
 and since $[\alpha,\alpha] \le \beta$, we have $a\mathrel\beta d(a, b, c)$.
 A symmetric argument shows $c\mathrel\alpha d(a, b, c)$, and so
 $a\mathrel\beta d(a, b, c) \mathrel\alpha c$. From this it follows
 that $\alpha$ and $\beta$ permute.
 (2) follows from (1) and Lemma~\ref{lemma:M3}
\end{proof}

\subsection*{Mal'tsev conditions and Mal'tsev classes}
A  collection  of varieties defined by a Mal'tsev condition is called a 
\emph{Mal'tsev class}.  Examples include the classes of congruence modular,
congruence distributive, and congruence semidistributive varieties. Other examples include
varieties with a Taylor term and varieties with a Hobby-McKenzie term. 
We are particularly concerned with varieties having a weak difference term.
By Kearnes and Szendrei \cite{KearnesSzendrei1998}
this is a Mal'tsev class; see 
\cite[Theorem 11.59]{FreeseMcKenzieMcNultyTaylor2022:a}. 
Having a weak difference term is a relatively weak property: all of the other 
conditions mentioned, except having a Taylor term, imply the existence of a weak difference term.
We are interested in varieties whose congruence variety is not the variety of all
lattices; that is, varieties that satisfy a nontrivial congruence identity.
These varieties form a Mal'tsev class; in fact it is the class of varieties with
a Hobby-McKenzie term \cite[Theorem A.2(11)]{KearnesKiss2013}.

An interval $I[\beta,\alpha]$ is \emph{abelian} if $\op C(\alpha,\alpha; \beta)$. 
This implies $[\alpha,\alpha] \le \beta$ but the converse is false.
However, if $\mathscr V$ has a weak difference term,  
$I[\beta,\alpha]$ is abelian if and only if $[\alpha,\alpha] \le \beta$.
Consequently, subintervals of abelian intervals are abelian in such varieties,
 \cite[Proposition 4.2]{Lipparini1994}; see also
 \cite[Theorem 11.29]{FreeseMcKenzieMcNultyTaylor2022:a}.

The 
\emph{solvable series for $\alpha$} is defined by 
$[ \alpha] ^0= \alpha$,
$ [\alpha] ^{n+1}= [[\alpha] ^{n},[\alpha] ^{n} ] $.
We say $\alpha$ is \emph{solvable} if $[\alpha] ^{n} = 0$ for some~$n$.
For $\beta \le \alpha$ we say the interval between $\beta $ and $\alpha $
is \emph{solvable} in case there exists some
$n$ such that $[ \alpha] ^{n} \leq \beta  $.  

The next theorem records some facts we need from \cite{Lipparini1994} on the
behavior of the commutator in varieties with a weak difference term.
The first is from \cite[Theorem 5.1]{Lipparini1994} and the second 
is from \cite[p. 197]{Lipparini1994}; see also
\cite[Theorem 11.87]{FreeseMcKenzieMcNultyTaylor2022:a}.

\begin{theorem}\label{wdt} 
Suppose that the algebra $\alg A$ has a weak difference term. Then
the following hold in $\Con(\alg A)$.
\begin{enumerate}[\normalfont (i)] 
 \item (Abelian and solvable intervals are preserved under transpositions)
If the interval $I[\beta,\alpha ]$ is abelian (solvable), 
then  the intervals
$I[\beta\meet\gamma ,\alpha \meet \gamma]$ and
 $I[\beta \join \delta ,\alpha \join \delta ]$ are abelian (solvable).
\item (Solvable intervals are intervals of permuting equivalence
 relations)
If the interval $I[\beta,\alpha ]$ is solvable and 
$ \gamma,\delta \in  I[\beta,\alpha ]$, then 
$ \gamma \circ \delta = \delta \circ \gamma $.
Hence $ I[\beta,\alpha ]$ is modular.
\end{enumerate}
\end{theorem}

We list in the next theorem the results we 
need from \cite{KearnesSzendrei1998,KearnesKiss2013}.
For $\alpha$, $\beta$ and $\gamma$ elements of a lattice, 
define $\beta^0 = \beta$, $\gamma^0 = \gamma$,
\begin{equation}\label{beta-m}
\beta^{m+1} = \beta \meet (\alpha \join \gamma^m) \qquad
\gamma^{m+1} = \gamma \meet (\alpha \join \beta^m)
\end{equation}

\begin{theorem}\label{kknontrivialid}
Suppose that $\mathscr V$ is a variety  and that 
the congruence variety  $\mathscr K$ associated with $\mathscr V$
is not the variety of all lattices. Then the following hold. 
\begin{enumerate}[\normalfont (i)] 
 \item\textup{(Kearnes and Szendrei
 \cite[Corollary 4.12]{KearnesSzendrei1998})} $\mathscr V$ 
has a weak difference term.
\item  
\textup{(Kearnes and Kiss 
\cite[Theorem 8.3]{KearnesKiss2013})} There exists a positive integer 
$m$ such that the congruence identity $ \beta ^m=\beta^{m+1} $ holds in  $\mathscr K$.
 \item
\textup{(Kearnes and Kiss 
\cite[Theorem 8.5]{KearnesKiss2013})}
Whenever $\alg A \in \mathscr V$ and  
$\alpha$, $\beta$, $\gamma$ are congruences 
of $\alg A$ such that  $\alpha \join\beta=\alpha\join\gamma$ then
the interval between $\alpha\join(\beta\meet\gamma)$ and
$\alpha\join\beta$ is abelian.
\end{enumerate} 
\end{theorem}

\section{Haiman's lattices and higher order Arguesian identities}

In \cite{Haiman1991} M. Haiman studied a chain of stronger
and stronger higher order Arguesian identities ($\text{D}_n$), given below.
He showed that
$\alg H_n(\alg F)$ witnesses that the congruence identities ($\text{D}_{n}$)
are properly increasing in strength. 
In Theorem~\ref{idequiv}
we present an identity, ($\text{D}_n^*$), equivalent to ($\text{D}_n$)
but which more closely resembles J\'onsson's Arguesian identity and is easier
to work with. 
The equivalence of these identities
is   of independent interest.

\begin{lemma}\label{ABx} 
Let $x_0$, $x_0'$, $A$, $B$ be elements of a modular
lattice such that $x_0\meet x_0'\leq B$ and $x_0 \vee x_0'\geq A$.
Then $A\leq x_0'\join(x_0\meet B)$ if and only if 
$x_0\meet(x_0'\join A)\leq B$.
\end{lemma}

\begin{proof}
$A\leq x_0'\join(x_0\meet B)$ is equivalent to 
$x_0'\join A\leq x_0'\join(x_0\meet B)$,
which implies 
$x_0\meet(x_0'\join A)\leq x_0\meet(x_0'\join(x_0\meet B))
=(x_0\meet x_0')\join (x_0 \meet B) \leq B$
by modularity.
The converse is the dual argument (with $B$ in place of $A$
and $x_0$ in place of $x_0'$).
\end{proof}

The next theorem generalizes a result first obtained
by A. Day and D. Pickering \cite{DayPickering1984} in the 
particular case $n=3$.
By Corollary~1 on page 104 of \cite{FreeseMcKenzieMcNultyTaylor2022}, ($\text{D}_3^*$)
is equivalent to the Agruesian identity.

\begin{theorem}\label{idequiv} 
Let $y_i=(x_i\vee x_{i+1}) \wedge  (x_i'\vee x_{i+1}')$
where the indices are computed modulo~$n$ so 
$y_{n-1} = (x_{n-1}\vee x_0) \wedge  (x_{n-1}'\vee x_0')$.
Then the following two equations are equivalent.
\begin{equation}
\tag{$\text{D}_n$}
x_0 \meet (x'_0 \join\bigwedge_{i=1}^{n-1}(x_i\vee x_i'))
\leq
x_1 \vee [(x_0'\vee x_1') \wedge \bigvee_{i=1}^{n-1} y_i]
\end{equation}
and
\begin{equation}
\tag{$\text{D}^*_n$}
\bigwedge_{i=0}^{n-1}(x_i\vee x_i')
\leq
x_0' \vee
(x_0 \wedge (x_1 \vee [(x_0'\vee x_1') \wedge \bigvee_{i=1}^{n-1} y_i]))
\end{equation}
\end{theorem}

\begin{proof} Both of these equations imply modularity:
for ($\text{D}_n^*$) we make the substitutions
$x_0 \mapsto x$, 
$x_i \mapsto y \join z$ for all~$i>0$, and
$x'_i \mapsto y \meet z$ for all~$i$. A similar substitution works 
for ($\text{D}_n$).

We claim
$x_0 \meet (x'_0 \join\bigwedge_{i=1}^{n-1}(x_i\vee x_i'))=
x_0 \meet (x'_0 \join\bigwedge_{i=0}^{n-1}(x_i\vee x_i'))$.
Indeed, modularity gives
\begin{align*}
 x_0 \meet \big(x'_0 \join\bigwedge_{i=0}^{n-1}(x_i\vee x_i')\big) 
     &= x_0 \meet \big(x'_0 \join [(x_0\join x_0')\meet\bigwedge_{i=1}^{n-1}(x_i\vee x_i')]\big) \\
     &= x_0 \meet (x_0 \join x'_0)\meet[x_0'\join \bigwedge_{i=1}^{n-1}(x_i\vee x_i')]\\
     &= x_0 \meet \big(x'_0 \join\bigwedge_{i=1}^{n-1}(x_i\vee x_i')\big)
\end{align*}

Now apply Lemma \ref{ABx} with
\[
A=\bigwedge_{i=0}^{n-1}(x_i\vee x_i') \quad\text{and}\quad
B=x_1 \vee [(x_0'\vee x_1') \wedge \bigvee_{i=1}^{n-1} y_i].
\]
Notice that $x_0\meet x_0'\leq B$, since 
$x_0 \meet x_0' \leq y_{n-1} = (x_{n-1} \join
x_0) \meet (x'_{n-1} \join x'_0)$, and clearly $x_0 \join x_0' \ge A$.
Thus the hypotheses of Lemma~\ref{ABx} hold. 
Using the claim we see that, for given $x_i$, $x_i'$, the conclusion of
that lemma is exactly the statement that ($\text{D}_n$) holds if
and only if ($\text{D}_n^*$) holds, as desired.
\end{proof}

Actually, the proof of Theorem \ref{idequiv} gives
something more: 

\begin{corollary}\label{idequiv2}
If\/ $\alg L$ is a modular lattice and
$x_i$, $x_i'\in \alg L$, then $x_i$, $x_i'$ satisfy
\textup{($\text{D}_n$)} if and only if $x_i$, $x_i'$ 
satisfy \textup{($\text{D}_n^*$)}.
\end{corollary}

Haiman showed that \textup{($\text{D}_n$)} holds in any lattice
of permuting equivalence relations,
but we require the following
strengthening of that fact.

\begin{lemma}\label{Dninperm} 
If\/ $\alg L$ is a sublattice of the lattice of equivalence relations on a set $A$, and if
$\alpha _i$, $\alpha _i'$ are elements of\/ $\alg L$,
and for each $i$, $\alpha _i$ and $\alpha _i'$ permute,
then the instance of\/ 
\textup{($\text{D}_n^*$)} in which $\alpha _i$ and $\alpha _i'$
are substituted for
$x_i$, $x_i'$,  $i = 0,\ldots,n-1$, holds in $\alg L$.
\end{lemma} 

\begin{proof} 
\begin{figure}[hbtp]
\centering
\begin{tikzpicture}[shorten >=1pt,node distance=2pt,inner
sep=1.5pt,auto,scale=2.8]
\node [circle,draw](u) at (-1,0){};
\node [circle,draw](z1) at (0,.1){};
\node [circle,draw](v) at  (1,0){};
\node [circle,draw](z0) at (0,-1){};
\node [circle,draw](z2) at (0,1){};
\node [circle,draw](z3) at (0,0.5){};
\node [circle,draw](z4) at (0,-0.2){};

\node [left=of u]{$a$};
\node [right=of v]{$b$};
\node [above=of z2]{$c_0$};
\node [below=of z0]{$c_{n-1}$};
\node [above=of z1]{$c_i$};
\node [above=of z3]{$c_1$};
\node [below=of z4]{$c_{i+1}$};

\draw (u) edge node[above=1pt]{$\alpha_i$} (z1)
          edge node[left=3pt]{$\alpha_{n-1}$} (z0)
          edge node[above=2pt]{$\alpha_0$} (z2)
          edge node[above=2pt]{$\alpha_1$} (z3)
          edge node[below=2pt]{$\alpha_{i+1}$} (z4)
          
      (z1) edge node[right=2pt]{$\gamma_i$} (z4)
          
      (v) edge node[above=2pt]{$\alpha'_0$} (z2)
          edge node[above=1pt]{$\alpha'_i$} (z1)
          edge node[right=3pt]{$\alpha'_{n-1}$} (z0)
          edge node[above=2pt]{$\alpha'_1$} (z3)
          edge node[below=2pt]{$\alpha'_{i+1}$} (z4);

\end{tikzpicture}
\caption{ \ }\label{fig:relations}
\end{figure}
Suppose that $\la a,b\ra\in
\bigwedge_{i=0}^{n-1}(\alpha_i\vee \alpha_i')$,
that is, $\la a,b\ra\in
\alpha_i\vee \alpha_i'$, for every~$i$.
Since $\alpha _i$ and $\alpha _i'$ permute for every 
$i$, this implies there exists $c_i$ such that 
$\la a, c_i\ra\in \alpha_i$ and
$\la c_i,b\ra\in \alpha_i'$.
Now let 
$\gamma_i = (\alpha _i\vee \alpha _{i+1}) \wedge (\alpha _i'\vee \alpha _{i+1}')$,
the element of $L$ corresponding to $y_i$. It follows
that $\la c_i,c_{i+1}\ra\in \gamma_i$. 
The indices are computed 
modulo $n$ so when $i = n-1$ we get 
$\la c_{n-1},c_0\ra\in \gamma_{n-1} = (\alpha _{n-1}\vee \alpha _0) \wedge
(\alpha _{n-1}'\vee \alpha _0')$.
These relations are indicated in Figure~\ref{fig:relations}.
Hence $c_1 \mathrel {\gamma_1} c_2 \mathrel{\gamma_2} c_3 \cdots c_{n-1}\mathrel{\gamma_{n-1}} c_0$,
and so $\la c_0,c_1\ra \in \gamma_1\join \cdots \join\gamma_{n-1}$. Thus
$\la c_0,c_1\ra \in (\alpha_0' \join \alpha_1')\meet\bigvee_{i=1}^{n-1} \gamma_i$
and so 
$\la a,c_0\ra\in \alpha_0\meet\big(\alpha_1 \join
   [(\alpha_0' \join \alpha_1')\meet\bigvee_{i=1}^{n-1} \gamma_i]\big)$.
Since $\la b, c_0\ra\in \alpha_0'$, we get
$\la b,a\ra\in 
\alpha_0' \vee
\big(\alpha_0\meet\big(\alpha_1 \join
   [(\alpha_0' \join \alpha_1')\meet\bigvee_{i=1}^{n-1} \gamma_i]\big)\big)$,
proving the lemma.
\end{proof} 

From Lemmas \ref{lem:M3-permute} and
\ref{Dninperm} we immediately get the following corollary.  

\begin{corollary}\label{cor:Dn-holds}
 Let $\alg A$ be an algebra
 with a weak difference term.
 Suppose $\alpha_i$ and $\alpha'_i \in \Con(\alg A)$, $i < n$, and also
 assume there is a congruence $\theta_i$ such that 
 $\alpha_i$, $\alpha'_i$ and $\theta_i$ are the atoms of a sublattice 
 of $\Con(\alg A)$ isomorphic to $\alg M_3$. Then \textup{($\text{D}_n^*$)}
 holds when $\alpha_i$ is substituted for $x_i$ and $\alpha'_i$ for $x'_i$.
\end{corollary}

If $\alg F$ is a skew field we let $\mathscr M_{\alg F}$ be the variety of left vector spaces
over~$\alg F$ and   $\mathscr M^{\text {fd}}_{\alg F}$ be the  class of all finite dimensional
left vector spaces over~$\alg F$.

\begin{lemma}\label{lem:HaimanLats}
 Let $\alg F$ be a field with $|F| > 2$ and let $\alg H_n(\alg F)$ be Haiman's lattice, $n\ge 3$.
 
\begin{enumerate}[\quad \normalfont(1)]
 \item Every proper sublattice of $\alg H_n(\alg F)$ can be embedded into the lattice
 of subspaces of a $2n$-dimensional vector space over $\alg F$.
 \item Every sublattice of $\alg H_n(\alg F)$ generated by less than $n$ elements is
 proper.
 \item There exist elements $x_i$ and $x_i'\in \alg H_n(\alg F)$, $i < n$, that generate $\alg H_n(\alg F)$
 and  witness the
 failure of 
 \textup{($\text{D}_n^*$)} in $\alg H_n(\alg F)$. Moreover, there are elements
 $x_i''\in \alg H_n(\alg F)$ such that $x_i$, $x_i'$ and $x_i''$ are the atoms of
 a sublattice isomorphic to $\alg M_3$. 
 \item 
 Let $\alg P$ be the prime subfield of $\alg F$. Then any nonprincipal ultraproduct of
 $\{\alg H_n(\alg F) : n = 3, 4, \ldots \}$ lies in $\SS\, \Con(\mathscr M_{\alg P})$.
\end{enumerate}
\end{lemma}

\begin{proof}
 (1) and (2) are Theorems 2 and 3  of \cite{Haiman1991}, respectively.
 Haiman's paper defines $x_i$ and $x_i'$ in $\alg H_n(\alg F)$ 
 to be atoms of an interval, $I[p_i,r_i]$,
 isomorphic to the lattice of subspaces of a three dimensional vector space over $\alg F$. 
 In such an interval the join of any two distinct  atoms contains a third atom and 
 (3) follows.
Haiman deals with the identity \textup{($\text{D}_n$)}, instead,
but we can equivalently use \textup{($\text{D}_n^*$)} in view
of Corollary \ref{idequiv2} and since Haiman's lattices $\alg H_n(\alg F)$
are modular. 
 
 $\SS\, \Con(\mathscr M_{\alg P})$ is a quasivariety, that is, it is defined by quasi-equations:
 by Lemma~\ref{ConV} it is closed under $\SS$ and $\PP$. In addition this class is 
 closed under ultraproducts; see \cite{HerrmannPoguntke1974} and \cite{MakkaiMcNulty1977}. See
 the proof of Theorem~2.1(2) of \cite{KearnesNation2008} for a very approachable proof.
 Consequently by Theorem~8.105 
 of \cite{FreeseMcKenzieMcNultyTaylor2022}, $\SS\, \Con(\mathscr M_{\alg P})$ is
 a quasivariety and so is defined by a set $\Phi$ of quasi-equations.
 (We point out that it is 
 not always the case that $\SS\, \Con(\mathscr V)$ is a 
 quasivariety for a variety $\mathscr V$: Kearnes and Nation \cite{KearnesNation2008} show, for example, that
 if $\mathscr V$ has a Taylor term but does not have a Hobby-McKenzie term, then 
 $\SS\, \Con(\mathscr V)$ is not a quasivariety. So for example the class of lattices embeddable
 into the congruence lattice of a semilattice is not elementary.)
 Let $\phi\in\Phi$. By (1) and (2) of this lemma
and the fact that the lattice of subspaces of a vector space over  $\alg F$ 
 is embedded into the lattice
 of subspaces of a vector space over its prime subfield,
 we see that if $\phi$ has at most $n$ variables, 
 then it holds in $\alg H_m(\alg F)$ for all $m > n$. So $\phi$ holds in all but
 finitely many $\alg H_m(\alg F)$'s. Thus $\phi$ holds in any nonprincipal 
 ultraproduct of the $\alg H_m(\alg F)$'s, proving (4).
\end{proof}

\section{Projectivity of $\alg M_3$ in congruence varieties}
 
In order to use the results from the previous sections to prove 
that no congruence variety, other than the variety of all lattices,
contains any $\alg H_n(\alg F)$ (that is to prove Theorem~\ref{thm:main}($1'$)), 
we will show that $\alg M_3$ is projective for every proper congruence variety.
To make this precise we need a definition. 
If $\alg L$ is any lattice, we say that
\emph{$\alg M_3$ is projective for $\alg L$}
if and only if whenever 
$\varphi :\alg L\onto \alg M_3$ 
is an epimorphism, then $\alg L$
contains a sublattice isomorphic to $\alg M_3$ 
which maps to $\alg M_3$ under $\varphi $.
We say that $\alg M_3$ is projective for a class $\mathscr X$ of 
lattices if 
$\alg M_3$ is projective for every member of $\mathscr X$.
If $\alg L$ does not have $\alg M_3$ as  homomorphic image, then 
$\alg M_3$ is projective for $\alg L$ vacuously. 
$\alg M_3$ is projective for the class of
modular lattices as follows from Dedekind's description of the 
free modular lattice on three generators, \cite{Dedekind1900}.
On the other hand, $\alg M_3$ is not projective for the class
of all lattices. This section will show that $\alg M_3$ is projective for 
every congruence variety except for the variety
of all lattices. 

\begin{theorem}\label{thm:m3projective}
Let $\mathscr K$ be the congruence variety of a variety $\mathscr V$. 
Then $\alg M_3$ is projective for $\mathscr K$ if and only if $\mathscr K$
is not the variety of all lattices.
\end{theorem}

\begin{proof}
First assume  $\mathscr K$ is the variety of all lattices. Then it contains the free lattice on three generations
which has a homomorphism onto $\alg M_3$. But $\alg M_3$ is not a sublattice of the 
the free lattice, see \cite{FreeseJezekNation1995},
so it is not projective for $\mathscr K$.

 Now assume that $\mathscr K$ is not the variety of all lattices. Let 
 $\alg L\in\mathscr K$ and suppose $\varphi :\alg L\onto \alg M_3$ 
is an epimorphism. By Lemma~\ref{ConV}(ii) there is a lattice 
$\alg L' \in \SS\, \alg{Con}(\alg A)$ and  an epimorphism
$\psi: \alg L'\onto \alg L$, for some $\alg A\in\mathscr V$. Summarizing:
\[ 
\alg L'\stackrel{\psi}{\twoheadrightarrow} \alg L    \stackrel{\varphi}{\twoheadrightarrow} \alg M_3
\quad
\text{with}
\quad
\alg L' \le \alg{Con}(\alg A).
\]
Let $a$, $b$ and $c$ be the atoms of $\alg M_3$. Choose 
$\alpha$, $\beta$ and $\gamma\in \alg L'$ to be pre-images of $a$, $b$ and $c$
under $\psi \comp \varphi$.

By Theorem~\ref{kknontrivialid}(ii) (and the symmetry between $\beta$ and $\gamma$) 
there is an $m$ such that $\beta^{m+1} = \beta^m$ and
$\gamma^{m+1} = \gamma^m$, where $\beta^n$ and $\gamma^n$ are defined by~\eqref{beta-m}.
Then
\[
\beta^m = \beta^{m+1} = \beta\meet (\alpha \join \gamma^m) \le \alpha \join \gamma^m,
\]
and so $\alpha \join\beta^m \le \alpha \join \gamma^m$, and by symmetry, 
$\alpha \join\beta^m = \alpha \join \gamma^m$.
Now an easy induction shows $(\phi\comp\varphi)(\beta^n) = b$ 
and $(\phi\comp\varphi)(\gamma^n) = c$ for all $n$. So, changing notation,
we may assume $\alpha$, $\beta$ and $\gamma$ are pre-images of $a$, $b$ and $c$ 
and that $\alpha\join\beta = \alpha\join\gamma$.
Now let $\alpha' = \alpha\join(\beta\meet\gamma)$ and note that $\alpha'$ is a pre-image
of $a$ and that $\alpha'\join\beta = \alpha'\join\gamma$ so, with another change in notation,
we may assume our pre-images  satisfy 
\[
\alpha\join\beta = \alpha\join\gamma \text{ \ and \ } 
\beta\meet\gamma \le \alpha.
\]
By Theorem~\ref{kknontrivialid}(iii) the interval $I[\alpha,\alpha\join\beta\join\gamma]$ is abelian.

By Theorem~\ref{kknontrivialid}(i) $\mathscr V$ has a weak difference term so by
Theorem~\ref{wdt}(i)  meeting with $\beta$ we have 
that $I[\alpha\meet\beta, (\alpha\join\beta\join\gamma)\meet\beta] = I[\alpha\meet\beta, \beta]$
is abelian. 
Similarly, $I[\alpha\meet\gamma, \gamma]$ is abelian.
Joining the first of these with $\alpha\meet\gamma$ yields 
that $I[(\alpha\meet\beta)\join(\alpha\meet\gamma), \beta\join(\alpha\meet\gamma)]$
is abelian. Joining the second with $\beta$ 
gives $I[\beta\join(\alpha\meet\gamma), \beta\join\gamma]$ is abelian. So we have
the chain
\[
(\alpha\meet\beta)\join(\alpha\meet\gamma) \le \beta\join(\alpha\meet\gamma) \le \beta\join\gamma,
\]
and it follows that the 
interval $I[(\alpha\meet\beta)\join(\alpha\meet\gamma), \beta\join\gamma]$ is solvable
and so by Theorem~\ref{wdt}(ii) it is modular. Let
\begin{align*}
\alpha' &= \alpha\meet(\beta\join\gamma), \\
\beta' &= \beta\join(\alpha\meet\gamma),\\
\gamma' &= \gamma\join(\alpha\meet\beta),
\end{align*}
and note these lie in the interval and are pre-images of $a$, $b$ and $c$, respectively.
Now the desired result follows from the projectivity of $\alg M_3$ in the variety of modular lattices.
Explicitly, defining 
\begin{align*}
\alpha'' &= (\alpha'\meet(\beta'\join\gamma'))\join(\beta'\meet\gamma') = \alpha'\join(\beta'\meet \gamma')\\
\beta'' &= (\beta'\meet(\alpha'\join\gamma'))\join(\alpha'\meet\gamma') = \beta'\meet(\alpha'\join\gamma')\\
\gamma'' &= (\gamma'\meet(\alpha'\join\beta'))\join(\alpha'\meet\beta') = \gamma'\meet(\alpha'\join\beta')
\end{align*}
we have that $\alpha''$, $\beta''$ and $\gamma''$ 
generate a sublattice which is isomorphic to $\alg M_3$ and maps onto $\alg M_3$ 
under $ \psi\comp\varphi$.
\end{proof}

\section{Proof of Theorem~\ref{thm:main}($\op{1}'$)}
Let $\alg L = \alg H_n(\alg F)$ be one of Haiman's lattices and suppose $\alg L$ lies in the congruence
variety $\mathscr K$ associated with a variety $\mathscr V$ and that $\mathscr K$
is not the variety of all lattices. 
By Lemma~\ref{ConV}(ii) $\alg L$ is a homomorphic image of a lattice $\alg L'$ which is a 
sublattice of $\alg{Con}(\alg A)$ for some $\alg A\in\mathscr V$. Let $\psi : \alg L' \onto \alg L$
be this epimorphism.
By Lemma~\ref{lem:HaimanLats}(3)
there  are elements $x_i$, $x_i'$ and $x_i''$,  $i < n$,  of $\alg L$ 
which are atoms of a 
sublattice  isomorphic to $\alg M_3$ such that
$x_i$ and $x_i'$ generate $\alg L$ and 
witness a failure of ($\text{D}^*_n$) in $\alg L$.
Since $\alg M_3$ is projective for $\mathcal K$ by 
Theorem~\ref{thm:m3projective}
there are elements $\overline{x_i}$, $\overline{x_i}'$  and   $\overline{x_i}''$ of $\alg L'$
which map under $\psi$ to $x_i$, $x_i'$ and $x_i''$ and are the atoms of
a sublattice isomorphic to $\alg M_3$.  Corollary~\ref{cor:Dn-holds} implies that ($\text{D}^*_n$) holds for 
these $\overline {x_i}$ and $\overline{x_i}'$, $i< n$. Applying $\psi$ we get
that ($\text{D}^*_n$) holds in $\alg L$ for $x_i$ and $x_i'$. This  contraction 
completes the proof.
\qed\medskip

As mentioned just before Theorem~\ref{idequiv},
($\text{D}^*_3)$ is equivalent to the arguesian law 
so, in a nondesaguesian projective plane, there are points $x_i$ and $x_i'$, $i = 0,1,2$, 
which witness the failure of   ($\text{D}^*_3)$. Since every line in a projective plane
has at least 3 points, there are points $x_i''$,  $i = 0,1,2$, such that $x_i$, $x_i'$ and $x_i''$
are the atoms of sublattice isomorphic to $\alg M_3$ in the lattice of subspaces of the
plane. Now using the arguments just above we get the following theorem.

\begin{theorem}
 Let $\alg L$ be the lattice of subspaces of a nonarguesian projective plane. Then 
 $\alg L$ lies in no congruence variety other than the variety of all lattices.\qed
\end{theorem}

The arguments above also give the following theorem which has a weaker hypothesis but also a
weaker conclusion. It 
should be compared to the concept
of omitted lattices of \S4.3 of \cite{KearnesKiss2013}.

\begin{theorem}
 Let $\alg L$ be one of Haiman's lattices or the lattice of subspaces of a nonarguesian projective plane.
 Let $\mathscr V$ be a variety with a weak difference term. Then $\alg L\notin \SS\,\alg{Con}(\mathscr V)$.\qed
\end{theorem}

\section{Lattices in $\SS\,\alg{Con}(\mathscr V)$}\label{sec:embed}

Before getting to the proof of Theorem~\ref{thm:main}($\text{2}'$) we prove some interesting 
embedding theorems. For example we show there is a large class of modular lattices, $\mathscr K_\infty$,
all of  which 
can be embedded into a member of
$\alg{Con}(\mathscr V)$
as long as $\mathscr V$ is not 
congruence meet semidistributive and has a weak difference term.

\begin{theorem}\label{thm:embedding}
 Let $\mathscr V$ be a variety having a weak difference term, and  assume
 that $\mathscr V$ is not congruence meet semidistributive.
 Then there is a prime field\/ $\alg P$ such that, for  every finite $n$, the
  lattice of subspaces of 
 a vector space over $\mathbf P$ of dimension $n$ lies in $\SS\,\alg{Con}(\mathscr V)$.
\end{theorem}

\begin{proof}
A variety is called \emph{congruence neutral} if $[\alpha,\beta] = \alpha\meet\beta$ holds in every
algebra in the variety. Equivalently  $[\alpha,\alpha] = \alpha$. 
 By \cite[Corollary 4.7]{KearnesSzendrei1998}, see also \cite[Theorem 11.37]{FreeseMcKenzieMcNultyTaylor2022:a},
 a variety is congruence neutral if and only if it is congruence meet semidistributive.
 Since we are assuming $\mathscr V$ is not congruence meet semidistributive, there is a congruence
 $\alpha $ on an algebra $\alg A\in\mathscr V$ with $[\alpha,\alpha] < \alpha$.
By \cite[Theorem 6.2]{CrawleyDilworth1973}
there is an element $\psi$  of $\op{Con}(\alg A)$ with $[\alpha,\alpha] \le \psi$
 but $\alpha\not\le \psi$, and such that 
 $\psi$ is completely meet irreducible with unique cover $ \psi\join\alpha$. By Theorem~\ref{wdt}(i)
 the interval $I[\psi,\psi\join\alpha]$ is abelian.  
 So, using the basic properties of the commutator (\cite[Lemma 11.4(viii)]{FreeseMcKenzieMcNultyTaylor2022:a}),
 $\alg A/\psi$ is subdirectly irreducible with an abelian monolith. 
 
 Changing notation, we may assume there is a subdirectly irreducible algebra $\alg A\in\mathscr V$ with
 an abelian monolith $\alpha$. Let $\alg A^n(\alpha)$ be the subalgebra of $\alg A^{n}$ with universe
 \[
 \{\la a_0, a_1, \ldots, a_{n-1}\ra \in A^{n} : a_i\mathrel\alpha a_j \text{ for all $i$ and $j$}\}.
 \]
 When $n=2$, $\alg A^2(\alpha) = \alg A(\alpha)$ which is described on pages 96--97
 in \cite{FreeseMcKenzieMcNultyTaylor2022}.
 Let $\overline\alpha \in\op{Con}(\alg A^n(\alpha))$ be such that
 $\la a_0,\ldots,a_{n-1}\ra \mathrel{\overline\alpha} \la b_0,\ldots,b_{n-1}\ra$  provided all these elements 
 are $\alpha$ related.
 Let $\eta_i$ be the kernel of the  $i^{\text{th}}$ projection of $\alg A^n(\alpha)$ onto $\alg A$.
 Of course $\alg A^n(\alpha)/\eta_i \iso \alg A$ and under this isomorphism
$\overline \alpha/\eta_i$ corresponds to~$\alpha$. 
Since $[\alpha,\alpha] = 0$ in $\alg A$, it follows using Lemma 11.4(viii) of \cite{FreeseMcKenzieMcNultyTaylor2022:a}
that $C(\overline\alpha,\overline\alpha;\eta_i)$ holds for each~$i$. Part (v) of that Lemma  gives
$C(\overline\alpha,\overline\alpha;\bigwedge_i\eta_i) = C(\overline\alpha,\overline\alpha ; 0)$.
Hence $\overline\alpha$ is an abelian congruence of $\alg A^n(\alpha)$.  Lemma~\ref{wdt}(ii)
implies the interval $I[0,\overline \alpha]$ in $\alg{Con}(\alg A^n(\alpha))$ is a modular lattice.
Let $\alg L_{n}$ denote this lattice.
Since $\eta_i\cov \overline\alpha$, $\alg L_{n}$ has length $n$ and its least element
is the meet of its coatoms. 
By (4.3) of \cite{CrawleyDilworth1973}, each $\alg L_n$ is a complemented modular lattice.

For $i\ne j$ the interval $I[\eta_i\meet\eta_j,\overline\alpha]$ has length 2. 
We claim this interval contains an element which is a complement of $\eta_i$
and of $\eta_j$ and so contains a sublattice isomorphic to $\alg M_3$.
We can prove this with $n=2$ (so $\alg A^2(\alpha) = \alg A(\alpha)$)
and then apply the Correspondence 
Theorem, \cite[Theorem 4.12]{McKenzieMcNultyTaylor1987}.
Let $\Delta$ be the congruence on $\alg A(\alpha)$ generated by 
\[
\{\la\la a,a\ra, \la b,b\ra\ra : a\mathrel\alpha b\}.
\]
Easy element-wise calculations show $\Delta\join\eta_i = \overline\alpha$,
$i = 0, 1$. Using the weak difference term it is not hard to show
$\Delta\meet\eta_i = \eta_0\meet\eta_1 = 0$. 
But actually these properties of $\Delta$ can be proved with the weaker
assumption that $\mathscr V$ has a Taylor term, as was shown
by Kearnes and Kiss; see 
\cite[Claim 3.25]{KearnesKiss2013}.

By elementary modular lattice theory the existence of these $\alg M_3$'s force
$\alg L_n$ to be simple. Classical coordination theorems of 
Artin \cite{Artin1940}, Birkhoff \cite{Birkhoff1948} and 
Frink \cite{Frink1946}, see Chapter~13 of 
\cite{CrawleyDilworth1973} and also \cite{JonssonMonk1969},
show that, for $n \ge 4$, $\alg L_n$ is isomorphic to the lattice of subspaces
of a vector space over a skew field~$\mathbf F$. 
Moreover, since $\alg L_{n-1}$ is isomorphic to  an interval (actually to a principal filter with 
generator $\eta_0\meet\eta_1\meet\cdots\meet\eta_{n-2}$) in $\alg L_n$ by the Correspondence Theorem,
the vector space for $\alg L_{n-1}$ is a subspace of the one for $\alg L_n$.
This implies the skew field $\alg F$ is the same for all~$n$.

If $\mathbf P$ is the prime subfield of $\mathbf F$ then the lattice
of subspaces of an $n$-dimensional vector space over $\mathbf P$
can be embedded 
via a cover-preserving map 
into the  lattice
of subspaces of an $n$-dimensional vector space over~$\mathbf F$.
This completes the proof of Theorem~\ref{thm:embedding}.
\end{proof}

\begin{remark}
\begin{enumerate}
\item
The embedding of Theorem~\ref{thm:embedding} can be assumed to be cover-preserving, as the proof shows.
\item
In \cite{FreeseHerrmannHuhn1981} it is shown that the congruence varieties $\VV\, \alg {Con}(\mathscr M_p)$,
$p$ a prime or $0$, associated with the variety of vector spaces over a prime field are minimal modular congruence
varieties and every nondistributive modular congruence variety contains one of these. Theorem~\ref{thm:embedding}
gives a clearer understanding of the situation. The minimal nondistributive congruence varieties 
are these and $\VV\, \alg {Con}(\mathscr P)$, the
congruence variety of Polin's variety; see \cite{DayFreese1980}.
If $\mathscr V$ is not congruence meet semidistributive, then its congruence variety contains
$\VV\, \alg {Con}(\mathscr M_p)$ for some~$p$.  If $\mathscr V$ is not congruence modular,
then it contains  $\VV\, \alg {Con}(\mathscr P)$. 
By Remark~\ref{rem:intersection} below the intersection of any two of the $\VV\, \alg{Con}(\mathscr M_p)$'s is
the variety of modular, 2-distributive lattices.
By the remarks of this paragraph,
this variety is not a congruence variety, so the intersection of two congruence
varieties need not be a congruence variety. On the othe hand, the join of two congruence
varieties is a congruence variety.
\end{enumerate}

\end{remark}

Recall that $\mathscr M_{\alg F}$ is 
the variety of vector spaces over $\alg F$ and
$\mathscr M_{\alg F}^\text{fd}$ is the class of all finite dimensional vector spaces
over $\alg F$.

If $\alg P$ is the prime field of characteristic $p$, a prime or 0, we let
$\mathscr M_p = \mathscr M_{\alg P}$
and $\mathscr M_p^\text{fd} = \mathscr M_{\alg P}^\text{fd}$.
Let
\[
\mathscr K_\infty =  \bigcap_{\text{$p$ a prime or $0$}} \mathscr \SS\,\alg{Con}(\mathscr M_p^\text{fd}).
\]
be the class of all  modular lattices that, for each $p$ a prime or $0$,
can be 
embedded into the lattice of subspaces of some finite dimensional
vector space over the prime field of characteristic~$p$.

\begin{corollary}\label{cor:k-infinity}
If\/ $\mathscr V$ is a variety with a weak difference term but which is not
congruence meet semidistributive, then
 $\mathscr K_\infty \subseteq \SS\, \alg{Con}(\mathscr V)$.\qed
\end{corollary}

$\mathscr K_\infty$ is a broad class that includes most modular lattices that have ever been pictured. 
Two members of this class are drawn in Figure~\ref{fig:examples}.
We present a nice  characterization of this class  due primarily  to C.~Herrmann, A.~Huhn, J.~B.\ Nation
and B.~J\'onsson. A lattice is 2-{\it distributive} if is satisfies the equation:
\begin{equation}\label{eq:2dist}
u\meet(x \join y \join z)  \approx (u\meet(x \join y)) 
\join (u\meet(x \join z))
\join (u\meet(y \join z)).
\end{equation}

\begin{figure}
\centering{
\begin{tikzpicture}[shorten >=1pt,node distance=2pt,inner
sep=1.5pt,auto,scale=0.7]
\node [circle,draw] (0) at (0,0){};
\node [circle,draw] (a) at (-1,1){};
\node [circle,draw] (b) at (0,1){};
\node [circle,draw] (c) at (1,1){};
\node [circle,draw] (bb) at (0,2){};

\draw (0)--(c)--(bb);
\draw (0)--(a)--(bb);
\draw (0)--(b)--(bb);

\node [circle,draw] (cc) at (1,2){};
\node [circle,draw] (e) at (2,2){};
\node [circle,draw] (f) at (1,3){};

\draw (c)--(cc)--(f);
\draw (c)--(e)--(f);
\draw (bb)--(f);

\node [circle,draw] (-22) at (-2,2){};
\node [circle,draw] (-12) at (-1,2){};
\node [circle,draw] (-13) at (-1,3){};
\node [circle,draw] (03) at (0,3){};
\node [circle,draw] (04) at (0,4){};

\draw (a)--(-22)--(-13)--(04)--(f);
\draw (bb)--(-13);
\draw (bb)--(03)--(04);


\node [circle,draw] (-12) at (-1,2){};
\node [circle,draw] (03) at (0,3){};
\node [circle,draw] (b) at (0,1){};

\draw (b)--(-12)--(03);

\color{black}
\node [circle,draw] (33) at (3,3){};
\node [circle,draw] (24) at (2,4){};
\node [circle,draw] (15) at (1,5){};
\node [circle,draw] (23) at (2,3){};
\node [circle,draw] (14) at (1,4){};
\draw (e)--(33)--(24)--(15)--(04);
\draw (24)--(f)--(14)--(15);
\draw (cc)--(23)--(14);

\node [circle,draw] (-33) at (-3,3){};
\node [circle,draw] (-24) at (-2,4){};
\node [circle,draw] (-15) at (-1,5){};
\node [circle,draw] (06) at (0,6){};
\node [circle,draw] (05) at (0,5){};
\node [circle,draw] (-14) at (-1,4){};
\node [circle,draw] (-23) at (-2,3){};

\draw (-22)--(-33)--(-24)--(-15)--(06)--(15);
\draw (06)--(05)--(04)--(-15);
\draw (-13)--(-24);
\draw (-12)--(-23)--(-14)--(05);
\draw (03)--(-14);

\node [circle,draw] (53) at (5,3){};
\node [circle,draw] (63) at (6,3){};
\node [circle,draw] (73) at (7,3){};
\node [circle,draw] (553) at (5.5,3){};
\node [circle,draw] (653) at (6.5,3){};
\node [circle,draw] (62) at (6,2){};
\node [circle,draw] (64) at (6,4){};

\draw (62)--(53)--(64);
\draw (62)--(63)--(64);
\draw (62)--(553)--(64);
\draw (62)--(653)--(64);
\draw (62)--(73)--(64);

\end{tikzpicture}
}
\caption{Two members of $\mathscr K_\infty$}\label{fig:examples}
\end{figure}
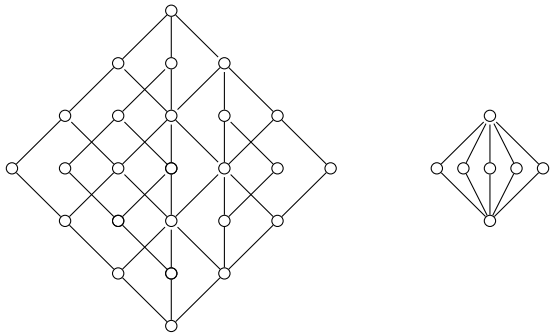

\begin{theorem}
 $\mathscr K_\infty$ is the class of all finite modular 2-distributive lattices.
\end{theorem}

\begin{proof}
 Suppose that $\alg L$ is a finite, modular 2-distributive lattice.
 In \cite{JonssonNation1987} J\'onsson and Nation show that, for every field $\alg F$, 
 $\alg L$ can be embedded via a covering preserving embedding 
  into the lattice of subspaces of an $\alg F$-vector space,
 provided $|\alg F| > |\alg L|$.
 Since the embedding is cover preserving and $\alg L$  is finite, we may assume these
 vector spaces are finite dimensional. This shows $\alg L\in \SS\,\alg{Con}(\mathscr M_p^{\text{fd}})$
 if $p=0$ (or if $p > |\alg L|$).
 For arbitrary  $p > 0$ we can choose a finite field $\alg F$ of characteristic $p$ with $|\alg F| > |\alg L|$.
 Let $\alg P$ be the prime subfield of $\alg F$. 
 Then an $\alg F$-vector space ${}_{\alg F}\alg V$
 has a reduct to
 a $\alg P$-vector space ${}_{\alg P}\alg V$ by restricting to scalar multiplication by elements of $\alg P$.
 Of course  $\dim {}_{\alg P}\alg V = \dim {}_{\alg F}\alg V \cdot \dim {}_{\alg P}\alg V$ and so is finite.
 This proves the reverse containment.

To see the opposite containment, let $\alg P$ and $\alg Q$ be prime fields of characteristics $p$ and
$q$, where $p\neq q$ are prime or $0$.  Suppose $\alg L\in\mathscr K_\infty$. Then there are
vector spaces ${}_{\alg P}\alg V$ and ${}_{\alg Q}\alg U$ such that $\alg L$ can be embedded into
the  lattice of subspaces of each.
If $\alg L$ is not $2$-distributive it contains a $2$-diamond (equivalently a $3$-frame) by
Huhn's Theorem \cite{Huhn1972}. Since $\alg L$  embeds into the subspace
lattice of ${}_{\alg P}\alg V$, this $2$-diamond ($3$-frame) has characteristic $p$, as 
defined in \cite{Freese1979:a} and \cite{HerrmannHuhn1975:a}. Since this holds for
${}_{\alg Q}\alg U$ as well, the frame has characteristic $q$ as well. 
But, as is shown in  \cite{Freese1979:a} and 
\cite{HerrmannHuhn1975:a}, this is impossible.
This contradiction completes the proof.
\end{proof}

\begin{remark}\label{rem:intersection}
 In \cite{HerrmannPickeringRoddy1994}    Herrmann, Pickering and Roddy generalized the
 result of J\'onsson and Nation  \cite{JonssonNation1987} by showing that, for every field $\alg F$, every modular
 $2$-distributive lattice can   be embedded into a vector space lattice over $\alg F$.  This, together with
 the projectivity of $n$-frames of characteristic $p$ \cite[Theorem 1.6]{Freese1979:a}, shows that the congruence 
 varieties associated with $\mathscr M_p$ and with $\mathscr M_q$, $p\ne q$,  intersect to the variety of 
 modular, $2$-distributive lattices. 
 Of course 
 the intersection of any modular variety of lattices with a semidistributive variety of lattices is distributive.
 In particular $\VV\,\alg{Con}(\mathscr M_p) \cap \VV\,\alg{Con}(\mathscr P)$, $p$ a prime or~$0$,
 is the variety of distributive lattices.

\end{remark}

\begin{remark}
In a recent paper \cite{AglianoBartaliFioravanti2023} Agliano, Bartali and Fioravanti show 
in their Theorem~1.4
that, if $\mathscr V$ is a 
 variety having a member whose congruence lattice has $\alg M_3$ as a sublattice such that the 
 congruences permute, 
 then $\SS\, \alg{Con}(\mathscr V)$ contains rods and snakes, which roughly
 are finite stacks of $\alg M_3$'s glued over one-dimensional intervals. $\alg M_{3,3}$ is an example.
 Moreover, in their concluding remarks the authors refer to a private communication by Keith Kearnes 
 that the hypothesis of their Theorem~1.4 is satisfied if $\mathscr V$ 
 is not congruence meet semidistributive.
 This result follows from  Corollary~\ref{cor:k-infinity}  but under the additional assumption
 that $\mathscr V$ has a weak difference term, and so it is a strengthening at least 
 regarding rods and snakes. On the other hand both the  lattices of Figure~\ref{fig:examples}
 lie in $\mathscr K_\infty$. One wonders if they lie in $\SS\,\alg {Con}(\mathscr V)$ whenever
 $\mathscr V$ is not congruence  meet semidistributive.
 \end{remark}

\section{Proofs of Theorem~\ref{thm:main}($\op{2}'$) and Theorem~\ref{thm:fbcv}}

We begin with a folklore result connecting the equational theory of algebraic lattices to that
of their compact elements.

\begin{lemma}
 An algebraic lattice satisfies a lattice equation if and only this equation holds 
 whenever compact elements are substituted for the variables.
\end{lemma}

\begin{proof}
This lemma implies that the ideal lattice of a lattice satisfies the same 
identities as the lattice and its proof is essentially the same.  But nevertheless
we give an outline.
Let $\alg L$ be an algebraic lattice and let $\alg L^{\text{c}}$ be the join subsemilattice of 
compact elements of $\alg L$. It is well known that  $\alg L$ is isomorphic to the lattice 
of ideals of $\alg L^c$ (join closed down-sets). It follows that if $c \le a_0 \join a_1$, with
$c$ compact, then $c\le c_0 \join c_1$ with $c_i$ compact and $c_i \le a_i$, $i =  0,1$.
More generally, an inductive argument on the complexity of a term $t$, shows that 
\[
\text{if $a_i\in L$,  $c$ is compact and }
c\le t(a_0,\ldots,a_{n-1}), \text{ then } c\le t(c_0,\ldots,c_{n-1})
\]
for some compact elements $c_i \le a_i$.

To prove the lemma suppose $s(x_0,\ldots,x_{n-1}) \approx t(x_0,\ldots,x_{n-1})$
holds for compact elements but fails in general.  
 Then, by symmetry, we may assume  there exist $a_i$'s such that $s(a_0,\ldots,a_{n-1}) \nleq t(a_0,\ldots,a_{n-1})$ in $\alg L$.
Consequently there is a compact element  $c$ with 
\begin{equation}\label{eq:compact}
c \leq s(a_0,\ldots,a_{n-1}), \text{ but } c \nleq t(a_0,\ldots,a_{n-1})
\end{equation}
Now by the  previous paragraph 
there are compact elements  $c_i\le a_i$ with $c \leq s(c_0,\ldots,c_{n-1})$.
But then
\[
c \leq s(c_0,\ldots,c_{n-1}) = t(c_0,\ldots,c_{n-1}) \leq t(a_0,\ldots,a_{n-1}),
\]
since $s = t$ holds for compact elements and lattice term functions are order-preserving.
This contradicts \eqref{eq:compact} and so proves the lemma.
\end{proof}

\begin{lemma}\label{lem:fdgenerate}
Let $\alg F$ and $\alg K$ be skew fields. 
\begin{enumerate}[\quad \normalfont(1)]
 \item
 The variety of lattices generated 
by $\alg{Con}(\mathscr M^{\text{\textup{fd}}}_{\alg F})$ equals 
the congruence variety of $\mathscr M_{\alg F}$. 
That is,
 \[
 \VV\, \alg{Con}(\mathscr M^{\text{\textup{fd}}}_{\alg F}) 
    = \VV\,\alg{Con}(\mathscr M_{\alg F}).
 \]
 \item
 $\alg F$ and $\alg K$ have the same characteristic if and only if
 \[
  \SS\,\alg{Con}(\mathscr M_{\alg F}) =  \SS\,\alg{Con}(\mathscr M_{\alg K}) .
 \]
 Moreover, if $\alg F$ and $\alg K$ have  different characteristics then
 \[
 \VV\,\alg{Con}(\mathscr M_{\alg F}) \nsubseteq  \VV\,\alg{Con}(\mathscr M_{\alg K}) .
 \]
 \end{enumerate}
\end{lemma}

\begin{proof}
 Let $\alg V$ be a vector space over $\alg F$ and let $\alg L =  \alg{Sub}(\alg V)$. 
 Of course, $ \alg{Con}(\alg V) \iso \alg{Sub}(\alg V)$
so $\alg L \in  \alg{Con}( \mathscr M_{\alg F})$. 
Since any finite set of compact elements of $\alg L$ is contained in $\alg W$ for
some finite dimensional vector space $\alg W$ of $\alg V$,
(1) follows from the previous lemma.

 For (2) first assume $\alg F \le \alg K$. Then 
 \begin{equation}\label{eq:containment}
 \SS\,\alg{Con}(\mathscr M_{\alg K}) \subseteq  \SS\,\alg{Con}(\mathscr M_{\alg F}).
 \end{equation}
 Indeed, if ${}_{\alg K}\!\alg V$ is a vector space over $\alg K$ then its reduct to
 $\alg F$, ${}_{\alg F}\!\alg V$, is a vector space over $\alg F$. 
 By Lemma~6.8 of \cite{FreeseMcKenzieMcNultyTaylor2022}, 
 $\alg{Con}({}_{\alg K}\!\alg V) \le \alg{Con}({}_{\alg F}\!\alg V)$, and hence
 $\alg{Con}({}_{\alg K}\!\alg V) \in \SS\,\alg{Con}(\mathscr M_{\alg F})$. 
 So \eqref{eq:containment} holds.
 
 Since  \eqref{eq:containment} is the only part of statement (2) of the lemma required
 in this paper, we will only sketch the rest of the proof.
 The opposite containment of \eqref{eq:containment},
 still assuming $\alg F \le \alg K$,
 can be proved using
 tensor  products: 
 if $\alg V$ is a vector space
 over $\alg F$, then $\alg V \otimes_{\alg F} \alg K$ is a vector space over $\alg K$
 and, since ${}_{\alg F}\!\alg K$ is flat, the map $\alg U \mapsto \alg U \otimes_{\alg F} \alg K$ 
 embeds the lattice of subspaces of $\alg V$, $\alg{Sub}(\alg V)$,
 into that of $\alg V \otimes_{\alg F} \alg K$.
 Thus, if $\alg P$ is the prime subfield of $\alg K$, we have 
 $\SS\,\alg{Con}(\mathscr M_{\alg P}) = \SS\,\alg{Con}(\mathscr M_{\alg K})$,
 and (2) easily follows.

 The last statement of (2) follows from \cite[Theorems 4 and 5]{HutchinsonCzedli1978}.
 \end{proof}

Let $\mathscr K$ be the congruence variety of a variety $\mathscr V$ and assume
$\mathscr K$ is not join semidistributive. 
To prove Theorem~\ref{thm:main}($\op{2}'$) we need to show
that there is a field $\alg F$ such that any nonprincipal 
ultraproduct of $\{\alg H_n(\alg F) : n \ge 3\}$ lies in $\mathscr K$.
This is trivial if $\mathscr K$ is the variety of all lattices so we may assume
$\mathscr V$ has a nontrivial (pure lattice) congruence identity. 
We know $\mathscr K$ is not join semidistributive  which implies $\mathscr V$
is not congruence join semidistributive, as was shown by Kearnes and Kiss in
\cite[Theorem~8.14 (1) $\Leftrightarrow$ (8)]{KearnesKiss2013}.
Now (1) $\Leftrightarrow$ (6) of that same theorem implies $\mathscr V$ is 
not congruence meet semidistributive.
Also, by the Kearnes-Szendrei result Theorem~\ref{kknontrivialid}(i) above, 
$\mathscr V$ has a weak difference term.

Theorem~\ref{thm:embedding} can
be interpreted to say that there is a prime field $\alg P$ such that
$\alg{Con}(\mathscr M^{\text{\textup{fd}}}_{\alg P})\subseteq
\SS\,\alg{Con}(\mathscr V)$.
Now by Lemma~\ref{lem:fdgenerate}(1) we have that the congruence variety of
$\mathscr V$,  namely $\mathscr K$, contains the congruence variety of $\mathscr M_{\alg P}$.
Let $\alg F$ be a field whose prime subfield is $\alg P$. (If  $\alg P$ is the two element field,
$\alg F$ should have at least 4 elements.)
By Lemma~\ref{lem:HaimanLats}(4) any nonprincipal ultraproduct
of $\{\alg H_n(\alg F) : n \ge 3\}$ lies in $\SS\, \alg{Con}(\mathscr M_{\alg P}) $
and so in $\mathscr K$.

This completes the proof of Theorem~\ref{thm:main}($\op{2}'$) and hence of Theorem~\ref{thm:main}.
\qed

To see Theorem~\ref{thm:fbcv} let 
$\mathscr V$ be a variety of algebras 
such that 
$\Con(\mathscr V)$ satisfies a nontrivial lattice identity and 
suppose $\alg{Con}(\mathscr V)$ is not semidistributive.
By the Kearnes-Kiss result cited above, this implies 
$\alg{Con}(\mathscr V)$ is not join semidistributive. 
By Theorem~\ref{thm:main}($\op{1}'$) and ($\op{2}'$)
there is a collection of lattices not in $\VV\, \alg{Con}(\mathscr V)$ 
whose ultraproduct is in. 
Consequently 
the congruence variety of $\mathscr V $ is not finitely based
by \cite[Theorem~8.52]{FreeseMcKenzieMcNultyTaylor2022}.
\qed

\subsection*{Data availability}
Data sharing not applicable to this article as datasets were neither generated nor analyzed.

\subsection*{Compliance with ethical standards}
The authors declare that they have no conflict of interest.
The first author belongs to the Editorial Board of Algebra Universalis.


\end{document}